\documentclass[11pt]{amsart}

\usepackage[english]{babel}
\usepackage[utf8x]{inputenc}
\usepackage{amsmath, amstext, amssymb, amsthm, amscd, graphicx, dsfont, manfnt, mathrsfs, marvosym,bm,fullpage}
\newtheorem{thm}{Theorem}

\newtheorem{lem}{Lemma}

\newcommand{\al}{\alpha}

\title{The monic Laguerre polynomials preserve real-rootedness}
\author{Praveen S. Venkataramana}
\date{\today}

\begin{document}

\begin{abstract}
Let $L_n(x)$ and $L_n^\alpha(x)$ be the $n$th Laguerre and associated Laguerre polynomial respectively. Fisk proved that the linear operator sending $x^n$ to $L_n(x)$ preserves real-rootedness. In this note we prove a stronger result; namely, that when $\alpha\ge 0$, the linear operator sending $x^n$ to $(-1)^n n! L_n^\alpha(x)$ preserves real-rootedness.
\end{abstract}

\maketitle

Let $\alpha$ be a nonnegative real number. The classical Laguerre polynomials are defined as follows:

$$L_n^\alpha(x) := \sum_{i=0}^n (-1)^i \binom{n+\alpha}{n-i} \frac{x^i}{i!}$$
Thus $(-1)^nn!L_n$ is monic for every $n$. Likewise, we define the \emph{scaled Hermite polynomials} as follows:

$$H_n^\xi(x) = e^{-\xi \ d^2/dx^2}x^n$$

In this note we prove the following:

\begin{thm}
If $N>0$ and $P(x) = \sum_{i=0}^N a_i x^i$ has all real roots, so does the polynomial $\sum_{i=0}^N (-1)^i i!a_iL_i(x)$.
\end{thm}

We first notice that:

$$ (-1)^n n! L_n^\alpha (x) = \exp\left(-x\frac{d^2}{dx^2} - (\alpha + 1)\frac{d}{dx}\right)x^n$$

To prove this, we define $\Lambda_x := x\frac{d^2}{dx^2} + (\alpha + 1)\frac{d}{dx}$ and note that $\Lambda_xx^k=k(k+\al)x^{k-1}$ for any integer $k$. Thus:

\begin{align*}
\exp\left(x\frac{d^2}{dx^2} + (\alpha + 1)\frac{d}{dx}\right)x^n &= \sum_{j=0}^\infty (-1)^j \frac{\Lambda_x^j}{j!}x^n\\
&= \sum_{j=0}^n \frac{(-1)^j}{j!}\left(\prod_{i=0}^{j-1} (n-i)(n+\al-i)\right)x^{n-j} \\
&= \sum_{j=0}^n \frac{(-1)^jn!(n+\al)!}{j!(n-j)!(n+\al-j)!}x^{n-j} \\
&= n! \sum_{i=0}^n (-1)^j \binom{n+\alpha}{j} \frac{x^{n-j}}{(n-j)!} = (-1)^nn!L_n^\alpha(x)
\end{align*}

The proof of Theorem 1 rests on a few important facts. First, the following properties are easily verified, when $\xi> 0,\alpha\ge 0$, and $k,\ell,m,n\in\mathbb{Z}_{\ge 0}$:

$$\int_{-\infty}^\infty H_k^\xi(x)H_\ell^\xi(x) e^{-x^2/4\xi}dx = 2k!\sqrt{\pi\xi}\delta_{k,\ell}$$

$$\int_0^\infty x^\alpha e^{-x} L_n^{(\alpha)}(x)L_m^{(\alpha)}(x)dx=\frac{\Gamma(n+\alpha+1)}{n!} \delta_{n,m}$$

It follows from the theory of orthogonal polynomials that both $H_k^\xi$ and $L_k^\alpha$ have $k$ distinct real roots, when $k>0$. We also need the following limit theorems, which are new as far as we can tell:

\begin{lem} Suppose $\xi\ne 0$, $k\in\mathbb{Z}_{>0}$, and $p$ is a polynomial so that $p(0)\ne 0$. Let $r_k(\xi)$ be the magnitude of the largest root of $H_k^\xi$. Then there is a positive real $h_0 = h_0(\xi,k,p,\alpha)$ such that when $0<h<h_0$, $e^{-h\Lambda_x}((x-\xi)^kp(x-\xi))$  has at least $k$ distinct roots in the interval $(\xi-2 r_k(\xi)\sqrt{h},\xi+2 r_k(\xi)\sqrt{h})$.
\end{lem}
\begin{proof} Let $h = \eta^2$ with $\eta$ nonnegative. Let $P_h(x) = e^{-h\Lambda_x}((x-\xi)^kp(x-\xi))$. Then we have:

\begin{align*}P_h(x+\xi)&= \exp -h\left((x+\xi)\frac{d^2}{dx^2} + (\alpha + 1)\frac{d}{dx}\right)(x^kp(x))\\
&= \exp\left(-h\xi\frac{d^2}{dx^2} - hx\frac{d^2}{dx^2} - h (\alpha + 1)\frac{d}{dx}\right)(x^kp(x))
\end{align*}

Thus:

\begin{align*}\frac{P_h(\epsilon\eta+\xi)}{\eta^k}&= \exp\left(-\xi\frac{d^2}{d\epsilon^2} - \eta\epsilon\frac{d^2}{d\epsilon^2} - \eta(\alpha + 1)\frac{d}{d\epsilon}\right)(\epsilon^kp(\epsilon\eta)) \\
&= \exp\left(-\xi\frac{d^2}{d\epsilon^2} - \eta\Lambda_\epsilon\right)(\epsilon^kp(\epsilon\eta))
\end{align*}

This latter function is a polynomial in $\eta$ of degree $k$ whose coefficients are polynomials in $\epsilon$, $\xi$ and $\alpha$ independent of $\eta$. Hence, if $|\epsilon| < 2r_k(\xi)$, there is a positive $\eta_0$ such that for $|\eta| < \eta_0$,

$$\left|\frac{P_h(\epsilon\eta+\xi)}{\eta^k} - p(0)H_k^\xi(\epsilon)\right| \le C\eta$$

for some constant $C$. Also note that the scaled Hermite polynomial $H_k^\xi$ has $k$ distinct real roots  in the interval $(-2r_k(\xi),2r_k(\xi))$, and thus there are $k+1$ numbers $a_1<a_2<\cdots<a_{k+1}$ in that interval so that the signs of $H_k^\xi(a_j)$ alternate, for $j=1,2,\dots,k+1$. Thus for small $\eta$, the signs of $P_h(a_j\eta+\xi)$ alternate as well, so by the intermediate value theorem, $P_h(x)$ has at least $k$ roots in the interval $(\xi-2 r_k(\xi)\sqrt{h},\xi+2 r_k(\xi)\sqrt{h})$. This proves the lemma.

\end{proof}

\begin{lem} Suppose $k\in\mathbb{Z}_{>0}$, and $p$ is a polynomial so that $p(0)\ne 0$. Let $s_k(\alpha)$ be the magnitude of the largest root of $L_k^\alpha$. Then there is a positive real $h_0 = h_0(k,p,\alpha)$ such that when $0<h<h_0$, $e^{-h\Lambda_x}(x^kp(x))$  has at least $k$ distinct roots in the interval $(-2 s_k(\alpha)h,2 s_k(\alpha)h)$.
\end{lem}
\begin{proof} The proof is almost identical to that of Lemma 1. Let $P_h(x) = e^{-h\Lambda_x}(x^kp(x))$. Then we have:

\begin{align*}\frac{P_h(\epsilon h)}{h^k}&= \exp\left( - \epsilon\frac{d^2}{d\epsilon^2} - (\alpha + 1)\frac{d}{d\epsilon}\right)(\epsilon^kp(\epsilon h)) \\
&= \exp\left(-\Lambda_\epsilon\right)(\epsilon^kp(\epsilon h))
\end{align*}

This latter function is a polynomial in $h$ of degree $k$ whose coefficients are polynomials in $\epsilon$ and $\alpha$ independent of $h$. Hence, if $|\epsilon| < 2s_k(\alpha)$, there is a positive $h_*$ such that for $|h| < h_*$,

$$\left|\frac{P_h(\epsilon h)}{h^k} - p(0)L_k^\alpha(\epsilon)\right| \le Ch$$

for some constant $C$. Also note that $L_k^\alpha(\epsilon)i$ has $k$ distinct real roots  in the interval $(-2s_k(\al),2s_k(\al))$, and thus there are $k+1$ numbers $a_1<a_2<\cdots<a_{k+1}$ in that interval so that the signs of $L_k^\al(a_j)$ alternate, for $j=1,2,\dots,k+1$. Thus for small $h$, the signs of $P_h(a_j\eta)$ alternate as well, so by the intermediate value theorem, $P_h(x)$ has at least $k$ roots in the interval $(-2 s_k(\alpha)h,2 s_k(\alpha)h)$. This proves the lemma.

\end{proof}

Now we are ready to prove Theorem 1. Fix a  non-constant real-rooted polynomial $P(x) = \sum_{i=0}^N a_i x^i$, and suppose that $$P(x) = \prod_{i=1}^n (x-\xi_i)^{m_i}$$ where $m_0,m_1,\dots,m_n\in\mathbb{Z}_{> 0}$, and $\xi_1<\xi_2<\cdots<\xi_n$ are real numbers. 

Lemmas 1 and 2 imply that for any $i$, there exists $h_i$ so that for all $h\in (0,h_i)$, the polynomial $e^{-h\Lambda_x}P(x)$ has $m_i$ roots in the interval $I_i(h) = (\xi_i-2 r_k(\xi_i)\sqrt{h},\xi_i+2 r_k(\xi_i)\sqrt{h})$ if $\xi_i\ne 0$, and $I_i(h) = (-2 s_k(\alpha)h,2 s_k(\alpha)h)$ if $\xi_i = 0$. Choose a positive real number $h'$ so that:
\begin{enumerate}
\item{$0<h'<h_i$ for all $i$, and}
\item{the intervals $I_i(h')$ are disjoint.}

\end{enumerate}
Then for every $h\in (0,h')$, $e^{-h\Lambda_x}P(x)$ has at least $m:= m_1+m_2+\cdots+m_n$ distinct roots. Since the degree of $e^{-h\Lambda_x}P(x)$ is that of $P(x)$, namely $m$, it follows that $e^{-h\Lambda_x}P(x)$ has $m$ real roots with multiplicity $1$.

We now define the set:

$$H(P) = \{h'\in (0,\infty)\ \ | \ \ e^{-h\Lambda_x}P(x)\textrm{ has $m$ simple real roots for all $h\in (0,h')$}\}$$

Then the above argument shows that $H(P)$ is nonempty and open in $(0,\infty)$. In addition, it is clear from the definition that $H(P)$ is either $(0,\infty)$ or an interval of the form $(0,y)$ for some positive real $y$. If the latter were true, then $e^{-h\Lambda_x}P(x)$ would have $m$ simple real roots for all $h\in (0,h')$, for all $h'<y$. However, since $\cup_{h'<y} (0,h') = (0,y)$, it follows that $y\in H(P)$, a contradiction. Hence $H(P) = (0,\infty)$.  However, we know that $P(x) = \sum_{i=0}^N a_i x^i$, so that

$$e^{-\Lambda_x}P(x) = \sum_{i=0}^N (-1)^i i!a_iL_i(x)$$

Since $1\in H(P)$, it follows that the latter polynomial has $N$ distinct real roots, and Theorem 1 is proven.

\section*{Bibliography}

1. Fisk, Steve. \emph{The Laguerre polynomials preserve real-rootedness}. \verb"arXiv:0808.2635"

2. Szwarc, Ryszard. \emph{Orthogonal Polynomials and Banach Algebras}. In \textit{Inzell Lectures on Orthogonal Polynomials}, edited by Wolfgang zu Castell, Frank Filbir, and Brigitte Forster.

\end{document}